\documentclass[a4paper,12pt]{article}
\usepackage[utf8]{inputenc}
\usepackage[a4paper,top=3cm,bottom=3cm,left=3cm,right=3cm,marginparwidth=1.75cm]{geometry}

\title{Connectivity Properties of McKay Quivers}
\author{Hazel Browne}
\date{March 2020}

\usepackage[parfill]{parskip}
\usepackage{amsmath}
\usepackage{amssymb}
\usepackage{amsthm}
\usepackage[colorlinks=true, allcolors=blue]{hyperref}
\usepackage{tikz}

\newcommand{\C}{\mathbb{C}}
\newcommand{\Q}{\mathbb{Q}}
\newcommand{\GL}{\textnormal{GL}}
\newcommand{\SU}{\textnormal{SU}}
\newcommand{\SL}{\textnormal{SL}}
\newcommand{\ot}{\otimes}

\newcommand{\A}{A_\rho(G)}
\newcommand{\la}{\left\langle}
\newcommand{\ra}{\right\rangle}
\newcommand{\res}{\downarrow}
\newcommand{\Ga}{\Gamma_\rho(G)}
\newcommand{\1}{\mathbf{1}}
\newcommand{\lb}{\left(}
\newcommand{\rb}{\right)}

\newtheorem{lemma}[subsection]{Lemma}
\newtheorem{proposition}[subsection]{Proposition}
\newtheorem{corollary}[subsection]{Corollary}

\theoremstyle{definition}
\newtheorem{definition}[subsection]{Definition}
\newtheorem{example}[subsection]{Example}

\theoremstyle{remark}

\begin{document}

\maketitle

\begin{abstract}
    We present several results regarding the connectivity of McKay quivers of finite-dimensional complex representations of finite groups, with no restriction on the faithfulness or self-duality of the representations. We give examples of McKay quivers, as well as quivers that cannot arise as McKay quivers, and discuss a necessary and sufficient condition for two finite groups to share a connected McKay quiver.
\end{abstract}

\section{Introduction}

Let $G$ be a finite group, and let $\sigma_1,\ldots,\sigma_r$ be the non-isomorphic irreducible representations of $G$ over $\C$, with $\sigma_1$ the trivial representation. Let $V$ be a finite-dimensional complex vector space and $\rho:G\to\GL(V)$ a representation of $G$. We define a matrix $A_\rho(G)=(a_{ij})_{i,j=1}^r$, where $a_{ij}$ are the non-negative integers such that
\[
\rho\ot\sigma_i=\bigoplus_{j=1}^r \sigma_{j}^{\oplus a_{ij}}.
\]
Call $\A$ the \textit{McKay matrix} of $G$. We further define $\Gamma_{\rho}(G)$ to be the quiver with vertices $v_1,\ldots,v_r$ and with $a_{ij}$ (directed) edges from $v_i$ to $v_j$. Call $\Ga$ the \textit{McKay quiver} of $G$. The vertices of $\Ga$ are in natural bijection with the irreducible representations of $G$, with $v_i$ the vertex corresponding to $\sigma_i$. By convention, we may replace a pair of opposite-pointing directed edges with a single undirected edge, and a quiver is said to be \textit{undirected} if it can be made undirected by this process.

The construction of $\A$ and $\Ga$ was introduced by John McKay in \cite{mckay}. McKay observed several properties of $\A$, most notably that its eigenvectors are the columns of the character table for $G$, with the character values of $\rho$ the corresponding eigenvalues. He went on to note that in the special case that $G$ is a finite subgroup of $\SU(2)$ and $\rho:G\hookrightarrow\GL_2(\C)$ the natural inclusion into $\GL_2(\C)$, $\Ga$ is one of the extended Dynkin diagrams of type $A$, $D$ or $E$. Moreover, this induces a bijection between the isomorphism classes of finite subgroups of $\SU(2)$ and such Dynkin diagrams, now known as the \textit{McKay correspondence}. As a corollary, this provides a description of the eigenvectors of the extended Cartan matrices, as these are of the form $2I-\A$.

To place McKay's result in historical context, it should be noted that it was already known that the (unextended) Dynkin diagrams of types $A$, $D$ and $E$ classified the \textit{Kleinian singularities}: singularities formed by quotienting $\C^2$ by the natural action of a finite subgroup of $\SU(2)$. Such singularities were studied by Klein \cite{klein} and Du Val \cite{duval}. That the representation-theoretic view gave the same bijection as the Kleinian singularities was an empirical observation, and various explanations of this connection have since been given, for example by Steinberg \cite{steinberg} and Gonzalez-Sprinberg and Verdier \cite{gonzalez}.

Others have generalised the correspondence, both from the representation-theoretic and algebro-geometric viewpoints. On the representation-theoretic side, Auslander and Reiten \cite{auslander} provide a generalisation to arbitrary two-dimensional representations, Happel, Preiser and Ringel \cite{happel} to arbitrary fields whose characteristics do not divide $|G|$, Butin and Perets \cite{butinPerets} to finite subgroups of $\SL_3(\C)$, and Butin \cite{butin} to finite subgroups of $\SL_4(\C)$.

However, there has been little investigation of the McKay quivers of non-faithful representations. While McKay \cite{mckay} states without proof that $\Ga$ is connected if and only if $\rho$ is faithful, and Happel, Preiser and Ringel \cite{happel} show that the vertices in the connected component containing $v_1$ are precisely the vertices corresponding to representations that factor through $G/\ker\rho$, both papers narrow their focus to the faithful case shortly thereafter.

In this paper, we generalise the above connectivity results. We show that strongly and weakly connected components of McKay quivers coincide (Proposition \ref{prop:connected-components}), we provide a necessary and sufficient condition for two vertices to lie in the same connected component (Proposition \ref{prop:connected-components}), and we determine how many connected components a McKay quiver has (Proposition \ref{prop:number-of-components}). We describe the connected component containing $v_1$, and in certain cases, the other connected components (Propositions \ref{prop:components-with-1d-rep} and \ref{prop:direct-product}). Finally we consider when two finite groups $G_1$ and $G_2$ can share a McKay quiver $\Gamma_{\rho_1}(G_1)\cong\Gamma_{\rho_2}(G_2)$ for some faithful representations $\rho_i$ of $G_i$ (Proposition \ref{prop:isographical-groups}).

Throughout, we provide examples of McKay quivers, connected components of McKay quivers, and quivers that cannot arise as McKay quivers.

\textbf{Acknowledgements.} This paper is based on an undergraduate project supervised by Anthony Henderson at the University of Sydney, and the author is extremely grateful for a year of patient help and guidance.

\section{Directedness and weightability}

For a quiver $Q$, write $A(Q)$ for the adjacency matrix of $Q$ (that is, the matrix whose $(i,j)$-entry is the number of arrows from the $i$th to the $j$th vertex of $Q$). For example, $A(\Ga)=\A$. All quivers are assumed to be finite.

Let $\chi_\rho$ be the character of $\rho$, and let $\chi_1,\ldots,\chi_r$ be the characters of $\sigma_1,\ldots,\sigma_r$ respectively. Write $\la-,-\ra$ for the inner product on class functions of $G$.

Since the classification of finite subgroups of $\SU(2)$ is known, it is possible to verify the McKay correspondence by direct computation. However, we can also use basic properties of McKay quivers to show that the McKay quiver of a finite subgroup of $\SU(2)$ must be one of the extended Dynkin diagrams of types $A$, $D$ or $E$, without making any direct calculations. To do this, we will use a characterisation of the extended Dynkin diagrams of types $A$, $D$ and $E$, for which we require the following definition.
\vspace{4mm}
\begin{definition} For a positive integer $k$, we say that a quiver $Q$ is \textit{$k$-weightable} if $A(Q)$ has a $k$-eigenvector whose components are positive integers. Such an eigenvector is called a \textit{$k$-weight vector} of $Q$. We can view a $k$-weight vector as a function on the vertices of $Q$, assigning to the $i$th vertex of $Q$ the $i$th component of the $k$-weight vector; call such a function a \textit{$k$-weighting} and call the values it takes \textit{weights}. We say that $Q$ is \textit{weightable} if it is $k$-weightable for some $k$.
\end{definition}

We now give a characterisation of the extended Dynkin diagrams of types $A$, $D$ and $E$ (see \cite[Theorem 2]{happel} for a proof).
\vspace{4mm}
\begin{proposition}
The extended Dynkin diagrams of types $A$, $D$ and $E$ are precisely the finite, connected, undirected quivers that are $2$-weightable.
\end{proposition}

The fact that the McKay quivers of finite subgroups of $\SU(2)$ are connected, undirected and $2$-weightable follows from more general properties of McKay quivers that we state in this section and the next.

We begin with $2$-weightability. This is a special case of the following observation due to McKay \cite{mckay}.
\vspace{4mm}
\begin{proposition}\label{prop:eigenvectors-of-A}
The eigenvectors of $\A$ are the columns of the character table for $G$, and the eigenvalue corresponding to a column is the value of $\chi_\rho$ in that column.
\end{proposition}

\begin{proof}
For each $i$, by definition of $\A=(a_{ij})_{i,j=1}^r$, we have
\[
\sum_{j=1}^r a_{ij}\chi_j=\chi_\rho\chi_i,
\]
and so for any group element $g\in G$, we have
\[
\A(\chi_j(g))_{j=1}^r=\lb\sum_{j=1}^r a_{ij}\chi_j(g)\rb_{i=1}^r=\chi_\rho(g)(\chi_i(g))_{i=1}^r.\qedhere
\]
\end{proof}
\vspace{4mm}
\begin{corollary}
The McKay quiver $\Ga$ is $(\dim\rho)$-weightable. In particular, if $\rho$ is the inclusion of a finite subgroup of $\SU(2)$, then $\Ga$ is $2$-weightable.
\end{corollary}
\begin{proof}
The column of the character table corresponding to the identity element of $G$ has positive integer entries, and by Proposition \ref{prop:eigenvectors-of-A} it is a $(\dim\rho)$-eigenvector of $\A$. In the case of finite subgroups of $\SU(2)$, $\dim\rho=2$ so we have $2$-weightability.
\end{proof}

A quiver is undirected if and only if its adjacency matrix is symmetric and the diagonal entries are even (to avoid directed self-loops). Symmetry of the adjacency matrix is easy to understand in the case of McKay quivers; this observation is also due to McKay \cite{mckay}.

Write $\rho^*$ for the dual representation of $\rho$.
\vspace{4mm}
\begin{proposition}\label{prop:symmetric}
The quivers $\Gamma_{\rho}(G)$ and $\Gamma_{\rho^*}(G)$ differ by a reversal of the arrows; that is, $\A^{T}=A_{\rho^*}(G)$. Moreover, $A_{\rho}(G)$ is symmetric if and only if $\rho$ is self-dual.
\end{proposition}

\begin{proof}
Writing $a_{ij}$ and $a_{ij}^*$ for the $(i,j)$-entries of $A_{\rho}(G)$ and $A_{\rho^*}(G)$ respectively, we have
\[
a_{ij}=\la \chi_\rho\chi_i,\chi_j\ra = \la \chi_j,\chi_\rho\chi_i\ra = \la \chi_{\rho^*}\chi_j,\chi_i\ra = a_{ji}^*.
\]
Thus if $\rho$ is self-dual then $A_{\rho}(G)$ is symmetric.

On the other hand, if $A_{\rho}(G)$ is symmetric then $a_{1j}=a_{j1}=a_{1j}^*$ for all $j$. But $a_{1j}$ and $a_{1j}^*$ are the multiplicities of $\sigma_j$ in $\rho$ and $\rho^*$ respectively, so $\rho$ and ${\rho}^*$ are isomorphic.
\end{proof}

If $\rho$ is the inclusion of a finite subgroup $G$ of $\SU(2)$ into $\GL_2(\C)$, then since $\SU(2)\subset\SL_2(\C)$, $\rho$ preserves the determinant -- a non-degenerate skew-symmetric bilinear form -- and this gives an isomorphism between $\rho$ and $\rho^*$. Thus $A_\rho(G)$ is symmetric.

It is possible to make a fairly direct argument for why the McKay quivers of finite subgroups of $\SU(2)$ must have no self-loops (and hence no directed self-loops), see \cite[1(4)(b)]{steinberg}. However, it also follows from the following more general proposition, which provides us with a broad class of undirected McKay quivers.
\vspace{4mm}
\begin{proposition}\label{prop:diagonal-entries}
If $\rho$ preserves a non-degenerate skew-symmetric bilinear form then the diagonal entries of $A_\rho(G)$ are even.
\end{proposition}

\begin{proof}
If $\rho$ preserves a non-degenerate skew-symmetric bilinear form $B$, then for each $i$, $\rho\ot \sigma_i\ot \sigma_i^*$ preserves the non-degenerate skew-symmetric bilinear form $B'$ given by
\[
B'(v\otimes w \otimes f, v'\ot w'\ot f') = B(v,v') f(w')f'(w).
\]
Invariant non-degenerate forms restrict to non-degenerate forms on the fixed-point subspace, which is the isotypic component of the trivial representation $\sigma_1$. Since non-degenerate skew-symmetric forms exist only on even-dimensional vector spaces, $\dim(\rho\ot\sigma_i\ot\sigma_i^*)_{[\sigma_1]}=\la \chi_\rho\chi_i\chi_i^*,\chi_{1}\ra=\la \chi_\rho\chi_i,\chi_i\ra=a_{ii}$ is even.
\end{proof}
\vspace{4mm}
\begin{example}
Suppose that $G$ is the binary dihedral group of order $12$, with presentation
\[
G=\la x,y,-1\mid x^2=y^3=(xy)^2=-1 \ra,
\]
where $-1$ denotes a central element of order $2$. Then $G$ has six irreducible representations $\sigma_1,\ldots,\sigma_6$ with corresponding characters $\chi_{1},\ldots,\chi_{6}$ shown in Figure \ref{fig:characters-of-bd12}.
\begin{figure}[h]
\begin{center}
  \begin{tabular}{|c | c c c c c c|}
    \hline
    class & $1$ & $-1$ & $x$ & $-x$ & $y$ & $y^2$ \\ \hline
    \# & $1$ & $1$ & $2$ & $2$ & $3$ & $3$\\ \hline
    $\chi_{1}$ & $1$ & $1$ & $1$ & $1$ & $1$ & $1$\\
    $\chi_{2}$ & $1$ & $1$ & $-1$ & $-1$ & $1$ & $1$ \\
    $\chi_{3}$ & $1$ & $-1$ & $i$ & $-i$ & $-1$ & $1$ \\
    $\chi_{4}$ & $1$ & $-1$ & $-i$ & $i$ & $-1$ & $1$ \\
    $\chi_{5}$ & $2$ & $2$ & $0$ & $0$ & $-1$ & $-1$ \\
    $\chi_{6}$ & $2$ & $-2$ & $0$ & $0$ & $1$ & $-1$ \\
    \hline
  \end{tabular}
\end{center}
\caption{The character table of the binary dihedral group of order $12$.}
\label{fig:characters-of-bd12}
\end{figure}

Figure \ref{fig:quivers-for-bd12} shows three McKay quivers of $G$ overlaid on the same vertex set: $\Gamma_{\sigma_3}(G)$, $\Gamma_{\sigma_5}(G)$, and $\Gamma_{\sigma_6}(G)$. Since $\sigma_3$ is not self-dual, $\Gamma_{\sigma_3}(G)$ is directed. Since $\sigma_5$ is self-dual, $\Gamma_{\sigma_5}(G)$ is invariant under edge reversal; however, it has directed self-loops. Since $\sigma_6$ preserves a non-degenerate skew-symmetric form (as can be verified by the Frobenius-Schur indicator), $\Gamma_{\sigma_6}(G)$ is undirected.

\begin{figure}[h]
    \centering
    \includegraphics[width=8cm]{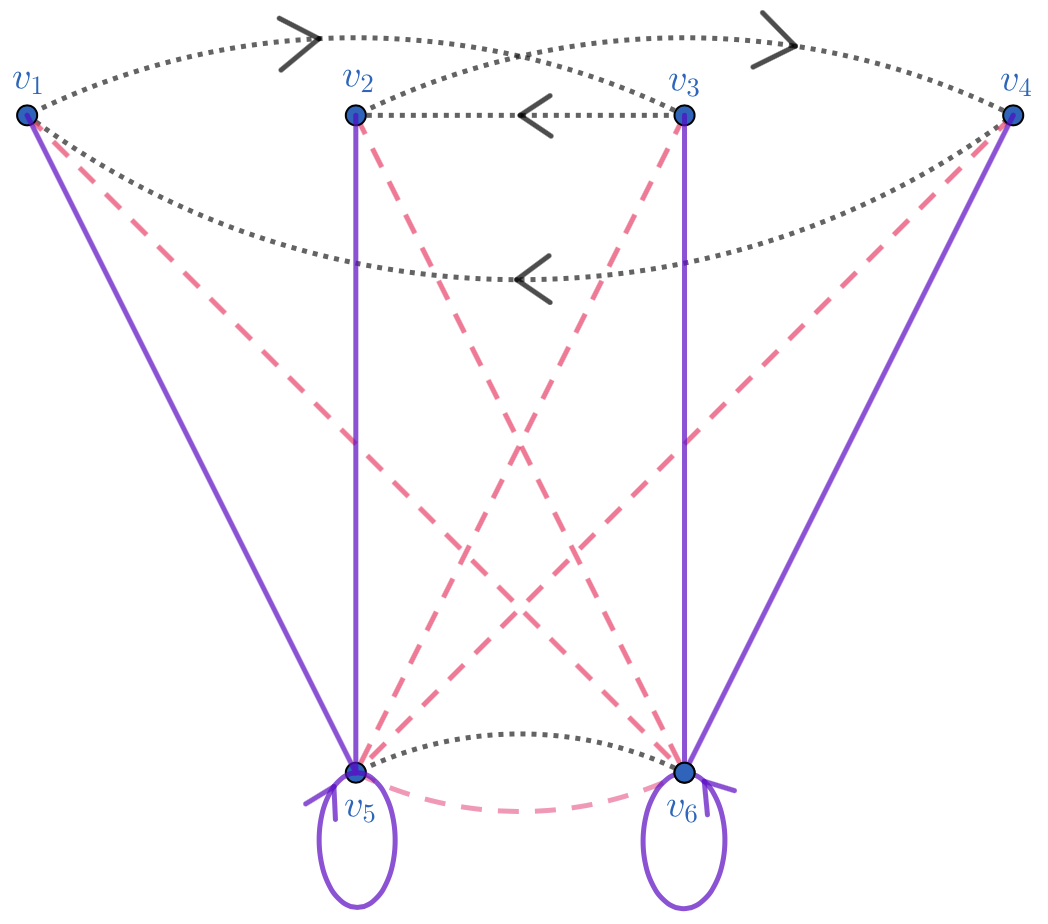}
    \caption{Three McKay quivers of the binary dihedral group of order $12$: $\Gamma_{\sigma_3}(G)$ is in dotted black; $\Gamma_{\sigma_5}(G)$ is in solid purple; $\Gamma_{\sigma_6}(G)$ is in dashed pink.}
    \label{fig:quivers-for-bd12}
\end{figure}

\end{example}

\section{Connectivity}

A priori, when considering directed McKay quivers, we must distinguish between the notions of \textit{strong connectivity} -- the existence of a path from $v_i$ to $v_j$ for any choice of vertices $v_i$ and $v_j$ -- and \textit{weak connectivity} -- the connectivity of the undirected graph that results from forgetting the direction of the edges. Since a union of two strongly/weakly connected subquivers with a vertex in common is strongly/weakly connected, \textit{strongly/weakly connected components} -- that is, maximal strongly/weakly connected subquivers -- split the vertices of a quiver into disjoint subsets.

However, we will ultimately see that strongly and weakly connected components of $\Gamma_{\rho}(G)$ coincide. Letting $N=\ker\rho$ and writing $\res_N$ for the restriction of a representation or character to $N$, we will show the following.
\vspace{4mm}
\begin{proposition}\label{prop:connected-components}
The following are equivalent:
\begin{enumerate}
    \item $v_i$ and $v_j$ lie in the same strongly connected component.
    \item $v_i$ and $v_j$ lie in the same weakly connected component.
    \item $\chi_i\res_N$ and $\chi_j\res_N$ are multiples of each other.
\end{enumerate}
\end{proposition}

To prove Proposition \ref{prop:connected-components}, and other results in this section, we will rely heavily on the following lemma.
\vspace{4mm}
\begin{lemma}\label{lem:walks}
There is a walk of length $L\geq 0$ from $v_i$ to $v_j$ in $\Ga$ if and only if $[\rho^{\otimes L}\otimes\sigma_i:\sigma_j]\geq 1$, where $[- : \sigma_j]$ denotes the multiplicity of $\sigma_j$ in a representation of $G$.
\end{lemma}
\begin{proof}
The proof is by induction, with the $L=0$ case trivial. A walk of length $L\geq 1$ from $v_i$ to $v_j$ is equivalent to a walk of length $L-1$ from $v_i$ to a vertex $v_k$ such that $[\rho\otimes \sigma_k:\sigma_j]\geq 1$. Using our inductive hypothesis, this is equivalent to the existence of $k$ such that $[\rho^{\otimes(L-1)}\otimes \sigma_i:\sigma_k]\geq 1$ and $[\rho\otimes \sigma_k:\sigma_j]\geq 1$, which is in turn equivalent to $[\rho^{\otimes L}\otimes \sigma_i:\sigma_j]\geq 1$.
\end{proof}

Let $\delta_{\1}$ be the class function whose value on the identity element $\1$ is $1$ and whose value on all other group elements is $0$. Our first application of Lemma \ref{lem:walks} was observed by McKay \cite{mckay}, but we will give a proof that we will be able to modify to prove a more general result (Proposition \ref{prop:characters-restricted-and-walks}).
\vspace{4mm}
\begin{proposition}
If $\rho$ is faithful then $\Ga$ is strongly connected.
\end{proposition}

\begin{proof}
If $\rho$ is faithful, then $\delta_{\1}=\sum\mu_{l}\chi_\rho^l$ for some finite collection of $\mu_l\in\C$ (this is due to Burnside \cite[XV.IV]{burnside}). Then for any $i,j$, we have
\[
\la \delta_{\1}\chi_i,\chi_j\ra = \sum\mu_l \la \chi_\rho^l\chi_i,\chi_j \ra.
\]
Since the left hand side is non-zero, $\la\chi_\rho^l\chi_i,\chi_j\ra$ is non-zero for some $l$, there is a walk from $v_i$ to $v_j$.
\end{proof}

In the non-faithful case, since $\rho$ descends to a faithful representation of $G/N$, we can replace $\delta_{\1}$ by $\delta_{N}$ in the above proof, giving the following:
\vspace{4mm}
\begin{proposition}\label{prop:characters-restricted-and-walks}
If $\chi_i\res_{N}$ and $\chi_j\res_{N}$ are multiples of each other, then there is a walk from $v_i$ to $v_j$.
\end{proposition}

The converse of the above proposition is also true, though we will require Clifford's Theorem \cite[Theorem 1]{clifford} to prove this.

To state Clifford's Theorem, let $\mathcal{G}$ be a finite group, $\mathcal{N}$ a normal subgroup of $\mathcal{G}$, and $\mathcal{S}_1,\ldots,\mathcal{S}_m$ the non-isomorphic irreducible representations of $\mathcal{N}$ over $\C$. There is a $\mathcal{G}$-action on $\{\mathcal{S}_1,\ldots,\mathcal{S}_m\}$ given by $(g\cdot \mathcal{S}_i )(n)=\mathcal{S}_i(g^{-1}ng)$.
\vspace{4mm}
\begin{proposition}[Clifford's Theorem]
Let $\mathcal{R}$ a be a finite-dimensional irreducible representation of $\mathcal{G}$ over $\C$. Then there is an orbit $\{\mathcal{S}_i\mid i\in I\}$ of the $\mathcal{G}$-action on $\{\mathcal{S}_1,\ldots,\mathcal{S}_m\}$ and a natural number $t$ such that 
\[
\mathcal{R} \res_\mathcal{N} \, \cong \left( \bigoplus_{i \in I} \mathcal{S}_i \right)^{\oplus t}.
\]
\end{proposition}
\vspace{4mm}
\begin{proposition}\label{prop:walks-characters-restricted}
If there is a walk from $v_i$ to $v_j$ then $\chi_i\res_{N}$ and $\chi_j\res_{N}$ are multiples of each other.
\end{proposition}

\begin{proof}
Let $\{\tau_1,\ldots,\tau_m\}$ be the non-isomorphic irreducible representations of $N$ over $\C$. By Clifford's Theorem, we have
\[
\sigma_i\res_N\cong\lb\bigoplus_{\alpha\in I_1}\tau_\alpha\rb^{\oplus t_1} \quad\textnormal{and}\quad \sigma_j\res_N\cong\lb\bigoplus_{\beta\in I_2}\tau_\beta\rb^{\oplus t_2}
\]
for some orbits $\{\tau_\alpha\mid \alpha\in I_1\}$ and $\{\tau_\beta\mid \beta\in I_2\}$ of the $G$-action on $\{\tau_1,\ldots,\tau_m\}$ and natural numbers $t_1,t_2$. These orbits are either disjoint, or $I_1=I_2$.

If there is a walk from $v_i$ to $v_j$ then for some $L\geq 0$ there is a subrepresentation of $\rho^{\ot L}\ot \sigma_i$ isomorphic to $\sigma_j$, by Lemma \ref{lem:walks}. Hence there is a subrepresentation of $(\rho^{\ot L}\ot \sigma_i)\res_N$ isomorphic to $\sigma_j\res_N$. Writing $\la-,-\ra_N$ for the inner product on class functions of $N$, we therefore have
\begin{equation}\label{ipn}
\la (\chi_\rho^L\chi_i)\res_N,\chi_j\res_N\ra_N \geq 1.
\end{equation}
Since $\chi_\rho^L\res_N$ is a multiple of the trivial character of $N$, we have
\[
\la \chi_i\res_N,\chi_j\res_N\ra_N\geq 1,
\]
so the orbits $\{\tau_\alpha\mid \alpha\in I_1\}$ and $\{\tau_\beta\mid \beta\in I_2\}$ cannot be disjoint. Thus $I_1=I_2$, so $\chi_i\res_N$ and $\chi_j\res_N$ are multiples of each other.
\end{proof}

Propositions \ref{prop:characters-restricted-and-walks} and \ref{prop:walks-characters-restricted} combine to prove Proposition \ref{prop:connected-components}.

\begin{proof}[Proof of Proposition \ref{prop:connected-components}]
Proposition \ref{prop:characters-restricted-and-walks} gives \textit{3}$\implies $\textit{1} and Proposition \ref{prop:walks-characters-restricted} gives \textit{2}$\implies$\textit{3}. Finally, \textit{1}$\implies$\textit{2} since strongly connected components are contained in weakly connected components.
\end{proof}

Since strongly and weakly connected components of $\Ga$ coincide, we will hereafter use the term \textit{connected component} to refer to either. Omitting $\rho$ and $G$ from our notation for conciseness, let $\Gamma_1,\ldots,\Gamma_s$ be the connected components of $\Ga$, with adjacency matrices $A_1,\ldots,A_s$ and vertex sets $\mathcal{V}_1,\ldots,\mathcal{V}_s\subset\{v_1,\ldots,v_r\}$ respectively.

We now turn to the question of how many connected components $\Ga$ has. Here, we will use a result from linear algebra known as the Perron-Frobenius Theorem, originally due to Perron \cite{perron} and extended from positive to non-negative matrices by Frobenius \cite{frobenius}.
\vspace{4mm}
\begin{definition}
An $n\times n$ matrix $X$ is called \textit{irreducible} if there is no permutation matrix $P$ such that $PXP^{-1}$ is block upper-triangular; equivalently in the case that the entries of $X$ are non-negative integers, $X$ is irreducible if the corresponding quiver is strongly connected.
\end{definition}
\vspace{4mm}
\begin{proposition}[Perron-Frobenius Theorem]
Let $X$ be a real irreducible $n \times n$ matrix all of whose entries are non-negative. Let $\lambda_1,\ldots,\lambda_n$ be the eigenvalues of $X$ and let $\lambda := \max |\lambda_i|_{1\leq i \leq n}$ be the spectral radius. Then $\lambda$ is an eigenvalue of $X$, and the corresponding eigenspace $E_\lambda$ is one-dimensional. Moreover, $E_\lambda$ contains a vector all of whose components are positive, and no other eigenspace of $X$ contains such a vector.
\end{proposition}
\vspace{4mm}
\begin{corollary}\label{cor:perron-frobenius}
If a strongly connected quiver $Q$ is $k$-weightable for some positive integer $k$, then the spectral radius of $A(Q)$ is $k$, its $k$-eigenspace is one-dimensional, and $Q$ is not $k'$-weightable for any $k'\neq k$.
\end{corollary}
\vspace{4mm}
\begin{proposition}\label{prop:number-of-components}
The number of connected components of $\Ga$ is equal to the number of conjugacy classes of $G$ contained in $N$.
\end{proposition}

\begin{proof}
By Proposition \ref{prop:eigenvectors-of-A}, the number of conjugacy classes of $G$ contained in $N$ is the multiplicity of $\dim\rho$ as an eigenvalue of $A_\rho(G)$.

Each connected component $\Gamma_i$ of $\Ga$ is $(\dim\rho)$-weightable by restricting the weighting $v_j\mapsto \dim\sigma_j$ to $\mathcal{V}_i$. Thus by Corollary \ref{cor:perron-frobenius}, the $(\dim\rho)$-eigenspace of $A_i$ is one-dimensional. Since $\A$ is a block diagonal matrix with blocks $A_1,\ldots,A_s$, the multiplicity of its $(\dim\rho)$-eigenspace is the sum of the multiplicities of the $(\dim\rho)$-eigenspaces of $A_1,\ldots,A_s$, which is the number of connected components of $\Ga$.
\end{proof}

\section{The connected components} Call the connected component containing $v_1$ the \textit{principal component} of $\Gamma_\rho(G)$. The representation $\rho$ descends to a representation $\rho_0:G/N\to\GL(V)$ of $G/N$, and it is clear that the principal component is just the McKay quiver $\Gamma_{\rho_0}(G/N)$.

The other connected components are harder to describe in general. In certain cases, this can be resolved using the following general property of McKay quivers.
\vspace{4mm}
\begin{lemma}\label{lem:action-of-dual-group}
The dual group $G^\vee$ of $G$ acts on $\Gamma_\rho(G)$ by quiver automorphisms, via an action that is simply transitive on the vertices corresponding to one-dimensional representations.
\end{lemma}

\begin{proof}
The dual group $G^\vee$ acts on the irreducible characters of $G$ by pointwise multiplication, and this action is simply transitive on $G^\vee$ since it is a group. Because the vertices of $\Ga$ are in correspondence with the irreducible characters of G, this defines an action of $G^\vee$ on the vertices of $\Ga$ that is simply transitive on the vertices corresponding to one-dimensional representations. This action is by quiver automorphisms since for $\chi\in G^\vee$ and $i,j\in\{1,\ldots,s\}$, we have
\[
\la \chi_\rho\chi_i,\chi_j \ra = \la \chi_\rho(\chi\chi_i),(\chi\chi_j) \ra.\qedhere
\]
\end{proof}
\vspace{4mm}
\begin{proposition}\label{prop:components-with-1d-rep}
Any connected component of $\Ga$ containing a vertex corresponding to a one-dimensional representation is isomorphic to the principal component. In particular, if $G$ is abelian then all the connected components of $\Ga$ are isomorphic.
\end{proposition}

\begin{proof}
Immediate from Lemma \ref{lem:action-of-dual-group}.
\end{proof}

Another notable case in which the other connected components are the same as the principal component is when $G$ is a direct product of $N$ and a complementary normal subgroup $H$.
\vspace{4mm}
\begin{proposition}\label{prop:direct-product}
Suppose that $G=N\times H$ for a complementary normal subgroup $H$. Then all the connected components of $\Ga$ are isomorphic.
\end{proposition}

\begin{proof}
If $\{\sigma_i\mid i\in I\}$ and $\{\sigma_j\mid j\in J\}$ are the sets of irreducible representations of $G$ that factor through $G/N$ and $G/H$ respectively, then a complete set of irreducible characters of $G$ is given by
\[
\{\chi_i\chi_j\mid i\in I,j\in J\}.
\]
Then using Proposition \ref{prop:connected-components}, the sets of irreducible characters corresponding to connected components are of the form $\{\chi_i\chi_j\mid i\in I\}$ for $j\in J$. Fix $j\in J$. Then for $i_1,i_2\in I$, we have
\[
\la \chi_{i_1}\chi_\rho,\chi_{i_2}\ra =\la\chi_{i_1}\chi_j\chi_{\rho},\chi_{i_2}\chi_j \ra.
\]
Thus multiplication of $\{\chi_i\mid i\in I\}$ by $\chi_j$ induces a quiver automorphism from the principal component to the connected component in correspondence with the irreducible characters $\{\chi_i \chi_j\mid i\in I\}$.
\end{proof}

\section{Reduced $(\dim\rho)$-weightings} In general, the other connected components may not be the same as the principal component. In particular, we will give an example of a connected component of a McKay quiver that does not arise as a McKay quiver in its own right (Example \ref{ex:binary-dihedral}).

However, this will require some method of showing that a given quiver is not the McKay quiver of any faithful representation. We know that McKay quivers of faithful representations are strongly connected and $(\dim\rho)$-weightable -- but so are connected components of McKay quivers. We will solve this problem by considering \textit{reduced weightings}.
\vspace{4mm}
\begin{definition}
For a quiver $Q$ and a positive integer $k$, a \textit{reduced $k$-weight vector} of $Q$ is a $k$-weight vector whose components share no non-trivial common factor. The corresponding $k$-weighting is a \textit{reduced $k$-weighting}.
\end{definition}

From the Perron-Frobenius Theorem, we see that a strongly connected $k$-weightable quiver has a unique reduced $k$-weight vector. In the case of McKay quivers of faithful representations, the reduced $(\dim\rho)$-weighting is just the dimension vector $(\dim\sigma_1,\ldots,\dim\sigma_r)$. Since a strongly connected $k$-weightable quiver is not $k'$-weightable for any $k'\neq k$, we observe:
\vspace{4mm}
\begin{proposition}\label{prop:dimensions-from-quiver}
If $\rho$ is faithful then we can determine $(\dim\sigma_1,\ldots,\dim\sigma_r)$ from $\Ga$ as its unique reduced $k$-weight vector for any $k$.
\end{proposition}

As a corollary, if a strongly connected quiver has a reduced weighting whose weights are not the multiset of dimensions of the irreducible representations of a finite group then it is not a McKay quiver. For example, the quiver in Figure \ref{pfig} is $3$-weightable with reduced weights $2$ and $3$; since $1$ (the dimension of the trivial representation) is not one of the reduced weights, the quiver is not a McKay quiver.

\begin{figure}[h]
    \centering
    \includegraphics[width=10cm]{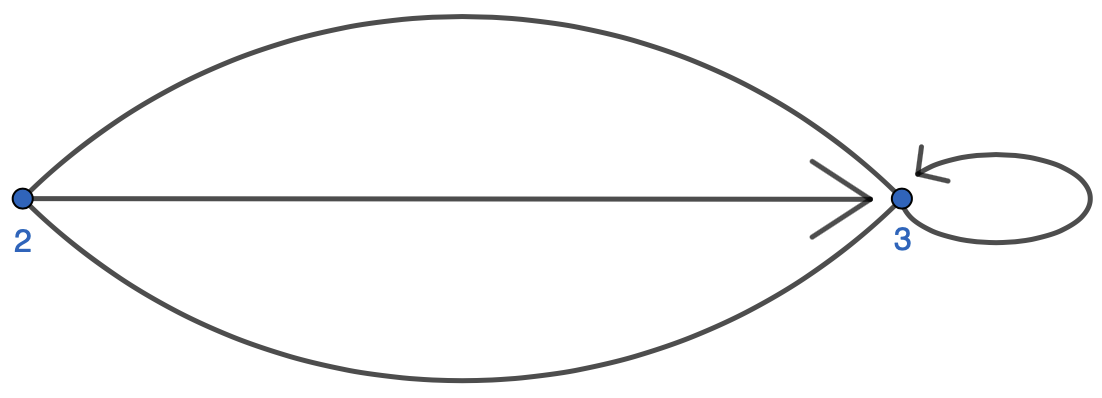}
    \caption{A $3$-weightable quiver with reduced weights $2$ and $3$. Since $\{2,3\}$ is not the multiset of dimensions of irreducible representations of any finite group, this quiver is not a McKay quiver.}
    \label{pfig}
\end{figure}

Proposition \ref{prop:dimensions-from-quiver} also allows us to determine which vertices of a McKay quiver correspond to one-dimensional representations, and by Lemma \ref{lem:action-of-dual-group}, there is a quiver automorphism that sends any such vertex to any other such vertex. Thus, if a strongly connected $k$-weightable quiver has two vertices of reduced weight $1$ and no quiver automorphism that sends one to the other, then this is not a McKay quiver; see Figure \ref{fig13} for an example.

\begin{figure}[h]
    \centering
    \includegraphics[width=10cm]{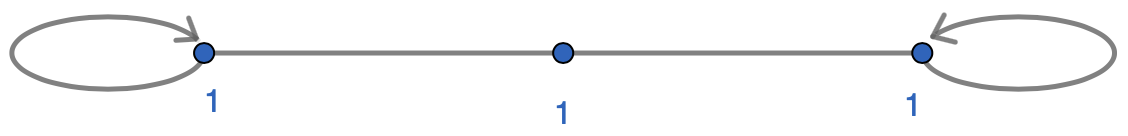}
    \caption{A $2$-weightable quiver with reduced weights $1$, $1$ and $1$. Since there is no quiver automorphism sending the centre vertex to either of the other vertices, this quiver is not a McKay quiver.}
    \label{fig13}
\end{figure}

We can now give an example of a McKay quiver with a connected component that is not itself the McKay quiver of a faithful representation.
\vspace{4mm}
\begin{example}\label{ex:binary-dihedral}
Consider the binary dihedral group of order $24$,
\[
G=\la x,y,-1\mid x^2=y^6=(xy)^2=-1 \ra,
\]
where $-1$ denotes a central element of order $2$. The irreducible characters $\chi_1,\ldots,\chi_9$ of $G$ are shown in Figure \ref{fig:char-tab-bd6}; let $\sigma_1,\ldots,\sigma_9$ be representations of $G$ with these characters. The McKay quiver $\Gamma_{\sigma_7}(G)$ is shown in Figure \ref{fig:disconnected-example-bd6}. The right-hand connected component is the $2$-weightable quiver in Figure \ref{fig13}, which is not a McKay quiver.

\begin{figure}[h]
\begin{center}
  \begin{tabular}{|c | c c c c c c c c c|}
    \hline
    class & $1$ & $-1$ & $x$ & $xy$ & $y$ & $y^2$ & $y^3$ & $y^4$ & $y^5$  \\ \hline
    \# & $1$ & $1$ & $6$ & $6$ & $2$ &$2$ & $2$ & $2$ & $2$ \\ \hline
    $\chi_{1}$ & $1$ & $1$ & $1$ & $1$ & $1$ & $1$ & $1$ & $1$ & $1$ \\
    $\chi_{2}$ & $1$ & $1$ & $1$ & $-1$ & $-1$ & $1$ & $-1$ & $1$ & $-1$  \\
    $\chi_{3}$ & $1$ & $1$ & $-1$ & $1$ & $-1$ & $1$ & $-1$ & $1$ & $-1$\\
    $\chi_{4}$ & $1$ & $1$ & $-1$ & $-1$ & $1$ & $1$ & $1$ & $1$ & $1$\\
    $\chi_{5}$ & $2$ & $-2$ & $0$ & $0$ & $\sqrt{3}$ & $1$ & $0$ & $-1$ & $-\sqrt{3}$ \\
    $\chi_{6}$ & $2$ & $-2$ & $0$ & $0$ & $-\sqrt{3}$ & $1$ & $0$ & $-1$ & $\sqrt{3}$\\
    $\chi_{7}$ & $2$ & $2$ & $0$ & $0$ & $1$ & $-1$ & $-2$ & $-1$ & $1$ \\
    $\chi_{8}$ & $2$ & $2$ & $0$ & $0$ & $-1$ & $-1$ & $2$ & $-1$ & $-1$\\
    $\chi_{9}$ & $2$ & $-2$ & $0$ & $0$ & $0$ & $-2$ & $0$ & $2$ & $0$\\
    \hline
  \end{tabular}
\end{center}
\caption{The character table of the binary dihedral group of order $24$.}
\label{fig:char-tab-bd6}
\end{figure}
\begin{figure}[h]
    \centering
    \includegraphics[width=12cm]{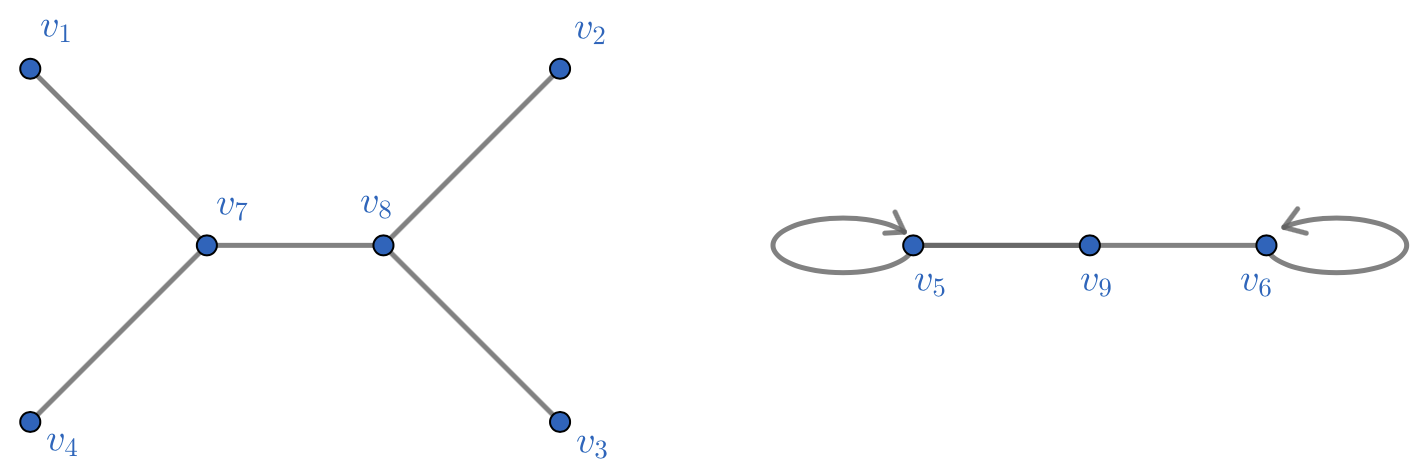}
    \caption{The disconnected McKay quiver $\Gamma_{\sigma_7}(G)$.}
    \label{fig:disconnected-example-bd6}
\end{figure}
\end{example}

Aside from allowing us to rule out various strongly connected weightable quivers as McKay quivers, Proposition \ref{prop:dimensions-from-quiver} tells us that two finite groups can only share a connected McKay quiver if their irreducible representations have the same dimensions. In fact, the converse is also true, as we see from the following observation.
\vspace{4mm}
\begin{proposition}\label{prop:regular-rep}
The McKay quiver of the regular representation of a finite group depends only on the multiset of dimensions of its irreducible representations.
\end{proposition}

\begin{proof}
Suppose that $\rho$ is the regular representation of $G$. Then for any $i,j$,
\[
x_{ij}=\la\chi_\rho\chi_i,\chi_j\ra = (\dim\sigma_i)(\dim\sigma_j).
\]

Thus any bijection between two McKay quivers of regular representations that preserves the dimension of the representation corresponding to each vertex is a quiver isomorphism.
\end{proof}

Since it is possible for two non-isomorphic groups to share the same multiset of dimensions of their irreducible representations -- for example, any two abelian groups of the same order -- Proposition \ref{prop:regular-rep} implies that non-isomorphic groups can have isomorphic McKay quivers.

Combining Proposition \ref{prop:dimensions-from-quiver} with Proposition \ref{prop:regular-rep}, we have:
\vspace{4mm}
\begin{proposition}\label{prop:isographical-groups}
Let $G_1$ and $G_2$ be finite groups. There exist faithful representations $\rho_1$ of $G_1$ and $\rho_2$ of $G_2$ such that $\Gamma_{\rho_1}(G_1)\cong\Gamma_{\rho_2}(G_2)$ if and only if the irreducible representations of $G_1$ and $G_2$ have the same dimensions (counted with multiplicity).
\end{proposition}

\begin{proof}
The forward implication is Proposition \ref{prop:dimensions-from-quiver}; the backward implication follows directly from Proposition \ref{prop:regular-rep}.
\end{proof}

\section{Weightable quivers that are not connected components}

As we have seen, there are connected components of McKay quivers that are not McKay quivers in their own right. Can any strongly connected weightable quiver be a connected component of a McKay quiver? We end this paper with an example that answers this question in the negative: there are strongly connected weightable quivers that cannot be connected components of McKay quivers.

In order to show that a strongly connected quiver $Q$ is not a connected component of a McKay quiver, we will look at the characteristic polynomial of $A(Q)$, making use of the following observation.
\vspace{4mm}
\begin{proposition}
The characteristic polynomial $f_{A_i}$ of the adjacency matrix $A_i$ of a connected component $\Gamma_i$ of a McKay quiver $\Ga$ is always solvable by radicals.
\end{proposition}

\begin{proof}
The eigenvalues of $A_\rho(G)$ are sums of roots of unity, since they are character values by Proposition \ref{prop:eigenvectors-of-A}. The eigenvalues of $A_i$ are a subset of the eigenvalues of $A_\rho(G)$; this is because any eigenvector of $A_i$ can be extended by zero to an eigenvector of $A_\rho(G)$.

Thus the splitting field of $f_{A_i}$ is contained in a field of the form $\Q(\zeta_1,\ldots,\zeta_m)$, where $\zeta_1,\ldots,\zeta_m$ are roots of unity, which is a radical extension of $\Q$, so $f_{A_i}$ solvable by radicals.
\end{proof}

We now give an example of a strongly connected weightable quiver whose adjacency matrix has a characteristic polynomial that is not solvable by radicals.
\vspace{4mm}
\begin{example}
Let $\gamma$ be the $8$-weightable undirected quiver shown in Figure \ref{fig:not-a-cc}, with the reduced $8$-weighting shown in blue.
\begin{figure}[h]
    \centering
    \includegraphics[width=12cm]{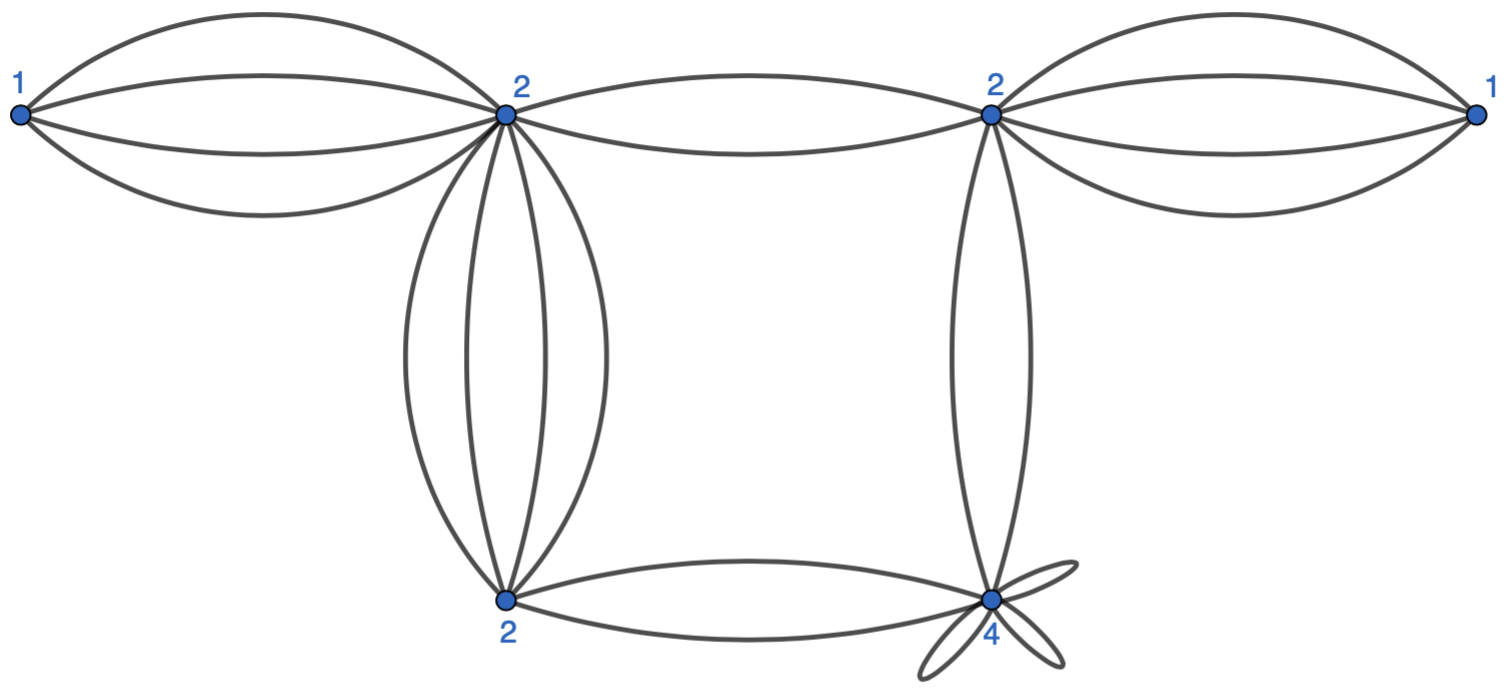}
    \caption{The $8$-weightable undirected quiver $\gamma$.}
    \label{fig:not-a-cc}
\end{figure}

The characteristic polynomial of the adjacency matrix of $\gamma$ is
\[
f_\gamma(x)=(x-8)(x^5+2x^4-44x^3-40x^2+400x+128).
\]
The Galois group of $x^5+2x^4-44x^3-40x^2+400x+128$ over $\Q$ is $S_5$, which is not solvable.

A polynomial is solvable by radicals if and only if its Galois group is a solvable. Hence, $f_\gamma$ is not solvable by radicals, and so $\gamma$ cannot be a connected component of a McKay quiver.
\end{example}

\bibliographystyle{plain}
\bibliography{bibliography}

\begin{thebibliography}{10}

\bibitem{auslander}
Maurice Auslander and Idun Reiten.
\newblock Mc{K}ay quivers and extended {D}ynkin diagrams.
\newblock {\em Trans. Amer. Math. Soc.}, 293(1):293--301, 1986.

\bibitem{burnside}
W.~Burnside.
\newblock {\em Theory of groups of finite order}.
\newblock Dover Publications, Inc., New York, 1955.
\newblock 2d ed.

\bibitem{butin}
F.~Butin.
\newblock Branching law for the finite subgroups of {${\bf SL}_4\Bbb{C}$} and
  the related generalized {P}oincar\'{e} polynomials.
\newblock {\em Ukrainian Math. J.}, 67(10):1484--1497, 2016.
\newblock Reprint of Ukra\"{\i}n. Mat. Zh. {{\bf{6}}7} (2015), no. 10,
  1321--1332.

\bibitem{butinPerets}
Fr\'{e}d\'{e}ric Butin and Gadi~S. Perets.
\newblock Branching law for finite subgroups of {$\bold{SL}_3\Bbb{C}$} and
  {M}c{K}ay correspondence.
\newblock {\em J. Group Theory}, 17(2):191--251, 2014.

\bibitem{clifford}
A.~H. Clifford.
\newblock Representations induced in an invariant subgroup.
\newblock {\em Annals of Mathematics}, 38(3):533--550, 1937.

\bibitem{duval}
Patrick Du~Val.
\newblock On isolated singularities of surfaces which do not affect the
  conditions of adjunction (part {I}.).
\newblock {\em Mathematical Proceedings of the Cambridge Philosophical
  Society}, 30(4):453–459, 1934.

\bibitem{frobenius}
Georg Frobenius.
\newblock {\"U}ber matrizen aus nicht negativen elementen.
\newblock {\em Königliche Gesellschaft der Wissenschaften}, 1524, 1912.

\bibitem{gonzalez}
G.~Gonzalez-Sprinberg and J.-L. Verdier.
\newblock Construction g\'{e}om\'{e}trique de la correspondance de {M}c{K}ay.
\newblock {\em Ann. Sci. \'{E}cole Norm. Sup. (4)}, 16(3):409--449 (1984),
  1983.

\bibitem{happel}
Dieter Happel, Udo Preiser, and Claus~Michael Ringel.
\newblock Binary polyhedral groups and {E}uclidean diagrams.
\newblock {\em Manuscripta Math.}, 31(1-3):317--329, 1980.

\bibitem{klein}
F.~Klein.
\newblock {\em Lectures on the Ikosahedron}.
\newblock Tr\"ubner \& Co, Ludgate Hill, London, 1888.
\newblock Translated from German by George Gavin Morrice.

\bibitem{mckay}
John McKay.
\newblock Graphs, singularities, and finite groups.
\newblock In {\em The {S}anta {C}ruz {C}onference on {F}inite {G}roups ({U}niv.
  {C}alifornia, {S}anta {C}ruz, {C}alif., 1979)}, volume~37 of {\em Proc.
  Sympos. Pure Math.}, pages 183--186. Amer. Math. Soc., Providence, R.I.,
  1980.

\bibitem{perron}
Oskar Perron.
\newblock Zur {T}heorie der {M}atrices.
\newblock {\em Math. Ann.}, 64(2):248--263, 1907.

\bibitem{steinberg}
Robert Steinberg.
\newblock Finite subgroups of {${\rm SU}_2$}, {D}ynkin diagrams and affine
  {C}oxeter elements.
\newblock {\em Pacific J. Math.}, 118(2):587--598, 1985.

\end{thebibliography}

\end{document}